\documentclass{article}

\usepackage{a4wide}

\usepackage{amsfonts}
\usepackage{amssymb}
\usepackage{amsthm}
\usepackage{amsmath}

\usepackage{graphicx}

\usepackage{url}

\newtheorem{tw}{Theorem}

\theoremstyle{definition}
\newtheorem{df}[tw]{Definition}
\newtheorem{ex}[tw]{Example}
\newtheorem{cor}[tw]{Corollary}

\begin{document}

\title{A Time-Scale Variational Approach\\
to Inflation, Unemployment and Social Loss\thanks{Part
of first author's Ph.D., which is carried out at the
University of Aveiro under the Doctoral Programme
\emph{Mathematics and Applications} 
of Universities of Aveiro and Minho.\newline
This is a preprint of a paper whose final and definitive form will appear 
in \emph{Control and Cybernetics}. Paper submitted 31-Oct-2012; 
revised 02-April-2013; accepted for publication 17-April-2013.}}

\author{Monika Dryl$^{1}$\\
\texttt{monikadryl@ua.pt}
\and
Agnieszka B. Malinowska$^{2}$\\
\texttt{a.malinowska@pb.edu.pl}
\and
Delfim F. M. Torres$^{1}$\\
\texttt{delfim@ua.pt}}

\date{$^1$CIDMA--Center for Research and Development in Mathematics and Applications,\\
Department of Mathematics, University of Aveiro, 3810-193 Aveiro, Portugal\\[0.3cm]
$^2$Faculty of Computer Science, Bialystok University of Technology,\\
15-351 Bia\l ystok, Poland}

\maketitle


\begin{abstract}

Both inflation and unemployment inflict social losses. When a tradeoff exists between the two,
what would be the best combination of inflation and unemployment?
A well known approach in economics to address this question
consists to write the social loss as a function of the rate of inflation $p$
and the rate of unemployment $u$, with different weights, and then, using known relations between $p$, $u$,
and the expected rate of inflation $\pi$, to rewrite the social loss function as a function of $\pi$.
The answer is achieved by applying the theory of the calculus of variations
in order to find an optimal path $\pi$ that minimizes the total social loss over a given time interval.
Economists dealing with this question use a continuous or a discrete variational problem.
Here we propose to use a time-scale model, unifying available results in the literature.
Moreover, the new formalism allow us to obtain new insights to the classical models
when applied to real data of inflation and unemployment.

\bigskip


\bigskip

\end{abstract}


\section{Introduction}

Time-scale calculus is a recent and exciting mathematical theory
that unifies two existing approaches to dynamic modelling --- difference and differential equations ---
into a general framework called dynamic models on time scales \cite{BohnerDEOTS,Hilger97,moz}.
Since a more general approach to dynamic modelling, it allows to consider more complex time domains,
such as $h\mathbb{Z}$, $q^{\mathbb{N}_0}$ or complex hybrid domains \cite{almeida:torres}.

Both inflation and unemployment inflict social losses. When a Phillips tradeoff exists between the two,
what would be the best combination of inflation and unemployment? A well-known approach
consists to write the social loss function as a function of the rate of inflation $p$
and the rate of unemployment $u$, with different weights;
then, using relations between $p$, $u$ and the expected rate of inflation $\pi$,
to rewrite the social loss function as a function of $\pi$;
finally, to apply the theory of the calculus of variations
in order to find an optimal path $\pi$ that minimizes
the total social loss over a certain time interval $[0,T]$ under study.
Economists dealing with this question implement the above approach using both
continuous and discrete models \cite{ChiangEDO,Taylor}.
Here we propose a new, more general, time-scale model.
We claim that such model describes better the reality.

We compare solutions to three models --- the continuous, the discrete, and the time-scale model with
$\mathbb{T}=h\mathbb{Z}$ --- using real data from the USA over a period of 11 years,
from 2000 to 2010. Our results show that the solutions to the classical continuous and discrete models
do not approximate well the reality. Therefore, while predicting the future, one cannot base
predictions on the two classical models only. The time-scale approach proposed here shows, however,
that the classical models are adequate if one uses an appropriate data sampling process.
Moreover, the proper times for data collection can be computed from the theory of time scales.

The paper is organized as follows. Section~\ref{prel} provides all the necessary definitions and results
of the delta-calculus on time scales, which will be used throughout the text. This section makes
the paper accessible to Economists with no previous contact with the time-scale calculus.
In Section~\ref{model} we present the economical model under our consideration, in continuous,
discrete, and time-scale settings. Section~\ref{main:results} contains our results.
Firstly, we derive in Section~\ref{main:theory} necessary (Theorem~\ref{cor1}
and Corollary~\ref{cor:ThZ}) and sufficient (Theorem~\ref{global}) optimality conditions
for the variational problem that models the economical situation. For the
time scale $\mathbb{T} = h\mathbb{Z}$ with appropriate values of $h$,
we obtain an explicit solution for the global minimizer
of the total social loss problem (Theorem~\ref{th:delf}). Secondly, we apply those conditions
to the model with real data of inflation \cite{rateinf} and unemployment
\cite{rateunemp} (Section~\ref{main:empirical}).
We end with Section~\ref{conclusions} of conclusions.


\section{Preliminaries}
\label{prel}

In this section we introduce basic definitions and theorems that will be useful in the sequel.
For more on the theory of time scales we refer to the gentle books
\cite{BohnerDEOTS,MBbook2001}. For general results on the calculus of variations on time scales
we refer the reader to \cite{Girejko,Malinowska,Martins} and references therein.

A time scale $\mathbb{T}$ is an arbitrary nonempty closed subset of $\mathbb{R}$.
Let $a,b\in\mathbb{T}$ with $a<b$. We define the interval $[a,b]$ in $\mathbb{T}$ by
$[a,b]_{\mathbb{T}}:=[a,b]\cap\mathbb{T}=\left\{t\in\mathbb{T}: a\leq t\leq b\right\}$.

\begin{df}[\cite{BohnerDEOTS}]
\label{def:jump:op}
The backward jump operator $\rho:\mathbb{T} \rightarrow \mathbb{T}$
is defined by $\rho(t):=\sup\lbrace s\in\mathbb{T}: s<t\rbrace$ for
$t\neq \inf\mathbb{T}$ and $\rho(\inf\mathbb{T}) := \inf\mathbb{T}$ if $\inf\mathbb{T}>-\infty$.
The forward jump operator $\sigma:\mathbb{T} \rightarrow \mathbb{T}$ is defined by
$\sigma(t):=\inf\lbrace s\in\mathbb{T}: s>t\rbrace$ for $t\neq \sup\mathbb{T}$
and $\sigma(\sup\mathbb{T}) := \sup\mathbb{T}$ if $\sup\mathbb{T}<+\infty$.
The backward graininess function $\nu:\mathbb{T} \rightarrow [0,\infty)$
is defined by $\nu(t):=t-\rho(t)$, while the forward graininess function
$\mu:\mathbb{T} \rightarrow [0,\infty)$ is defined by $\mu(t):=\sigma(t)-t$.
\end{df}

\begin{ex}
The two classical time scales are $\mathbb{R}$ and $\mathbb{Z}$,
representing the continuous and the purely discrete time, respectively.
The other example of interest to the present study is the periodic time scale
$h\mathbb{Z}$. It follows from Definition~\ref{def:jump:op} that
if $\mathbb{T}=\mathbb{R}$, then
$\sigma (t)=t$, $\rho(t)=t$, and $\mu(t) = 0$ for all $t \in \mathbb{T}$;
if $\mathbb{T}=h\mathbb{Z}$, then $\sigma(t)= t+h$, $\rho(t)= t-h$,
and $\mu(t) = h$ for all $t\in\mathbb{T}$.
\end{ex}

A point $t\in\mathbb{T}$ is called \emph{right-dense},
\emph{right-scattered}, \emph{left-dense} or \emph{left-scattered}
if $\sigma(t)=t$, $\sigma(t)>t$, $\rho(t)=t$,
and $\rho(t)<t$, respectively. We say that $t$ is \emph{isolated}
if $\rho(t)<t<\sigma(t)$, that $t$ is \emph{dense} if $\rho(t)=t=\sigma(t)$.


\subsection{The delta derivative and the delta integral}

We collect here the necessary theorems and properties
concerning differentiation and integration on a time scale.
To simplify the notation, we define $f^{\sigma}(t):=f(\sigma(t))$.
The delta derivative is defined for points in the set
$$
\mathbb{T}^{\kappa} :=
\begin{cases}
\mathbb{T}\setminus\left\{\sup\mathbb{T}\right\}
& \text{ if } \rho(\sup\mathbb{T})<\sup\mathbb{T}<\infty,\\
\mathbb{T}
& \hbox{ otherwise}.
\end{cases}
$$

\begin{df}[Section~1.1 of \cite{BohnerDEOTS}]
We say that a function $f:\mathbb{T}\rightarrow\mathbb{R}$ is
\emph{$\Delta$-differentiable} at
$t\in\mathbb{T}^\kappa$ if there is a number $f^{\Delta}(t)$
such that for all $\varepsilon>0$ there exists a neighborhood $O$
of $t$ such that
$$
|f^\sigma(t)-f(s)-f^{\Delta}(t)(\sigma(t)-s)|
\leq\varepsilon|\sigma(t)-s|
\quad \mbox{ for all $s\in O$}.
$$
We call to $f^{\Delta}(t)$ the \emph{$\Delta$-derivative} of $f$ at $t$.
\end{df}

\begin{tw}[Theorem~1.16 of \cite{BohnerDEOTS}]
\label{differentiation}
Let $f:\mathbb{T} \rightarrow \mathbb{R}$
and $t\in\mathbb{T}^{\kappa}$. The following holds:
\begin{enumerate}
\item
If $f$ is differentiable at $t$, then $f$ is continuous at $t$.

\item
If $f$ is continuous at $t$ and $t$ is right-scattered,
then $f$ is differentiable at $t$ with
$$
f^{\Delta}(t)=\frac{f^\sigma(t)-f(t)}{\mu(t)}.
$$

\item
If $t$ is right-dense, then $f$ is differentiable at $t$
if, and only if, the limit
$$
\lim\limits_{s\rightarrow t}\frac{f(t)-f(s)}{t-s}
$$
exists as a finite number. In this case,
$$
f^{\Delta}(t)=\lim\limits_{s\rightarrow t}\frac{f(t)-f(s)}{t-s}.
$$

\item
If $f$ is differentiable at $t$, then
$f^\sigma(t)=f(t)+\mu(t)f^{\Delta}(t)$.
\end{enumerate}
\end{tw}

\begin{ex}
If $\mathbb{T}=\mathbb{R}$, then item~3
of Theorem~\ref{differentiation} yields that
$f:\mathbb{R} \rightarrow \mathbb{R}$
is delta differentiable at $t\in\mathbb{R}$ if, and only if,
$$
f^\Delta(t)=\lim\limits_{s\rightarrow t}\frac{f(t)-f(s)}{t-s}
$$
exists, i.e., if, and only if, $f$ is differentiable (in the ordinary sense) at $t$:
$f^{\Delta}(t)=f'(t)$. If $\mathbb{T}=h\mathbb{Z}$, then point~2
of Theorem~\ref{differentiation} yields that
$f:\mathbb{Z} \rightarrow \mathbb{R}$ is delta differentiable
at $t\in h\mathbb{Z}$ if, and only if,
\begin{equation}
\label{eq:delta:der:h}
f^{\Delta}(t)=\frac{f(\sigma(t))-f(t)}{\mu(t)}=\frac{f(t+h)-f(t)}{h}.
\end{equation}
In the particular case $h=1$, $f^{\Delta}(t)=\Delta f(t)$,
where $\Delta$ is the usual forward difference operator.
\end{ex}

\begin{tw}[Theorem~1.20 of \cite{BohnerDEOTS}]
\label{tw:differpropdelta}
Assume $f,g:\mathbb{T}\rightarrow\mathbb{R}$
are $\Delta$-differentiable at $t\in\mathbb{T^{\kappa}}$. Then,
\begin{enumerate}
\item The sum $f+g:\mathbb{T}\rightarrow\mathbb{R}$ is
$\Delta$-differentiable at $t$ with
$(f+g)^{\Delta}(t)=f^{\Delta}(t)+g^{\Delta}(t)$.

\item
For any constant $\alpha$, $\alpha f:\mathbb{T}\rightarrow\mathbb{R}$
is $\Delta$-differentiable at $t$ with
$(\alpha f)^{\Delta}(t)=\alpha f^{\Delta}(t)$.

\item The product $fg:\mathbb{T} \rightarrow \mathbb{R}$
is $\Delta$-differentiable at $t$ with
\begin{equation*}
(fg)^{\Delta}(t)=f^{\Delta}(t)g(t)+f^{\sigma}(t)g^{\Delta}(t)
=f(t)g^{\Delta}(t)+f^{\Delta}(t)g^{\sigma}(t).
\end{equation*}

\item If $g(t)g^{\sigma}(t)\neq 0$,
then $f/g$ is $\Delta$-differentiable at $t$ with
\begin{equation*}
\left(\frac{f}{g}\right)^{\Delta}(t)
=\frac{f^{\Delta}(t)g(t)-f(t)g^{\Delta}(t)}{g(t)g^{\sigma}(t)}.
\end{equation*}
\end{enumerate}
\end{tw}

\begin{df}[Definition~1.71 of \cite{BohnerDEOTS}]
A function $F:\mathbb{T} \rightarrow \mathbb{R}$ is called
an antiderivative of $f:\mathbb{T} \rightarrow \mathbb{R}$ provided
$F^{\Delta}(t)=f(t)$ for all $t\in\mathbb{T}^{\kappa}$.
\end{df}

\begin{df}[\cite{BohnerDEOTS}]
A function $f:\mathbb{T} \rightarrow \mathbb{R}$ is called rd-continuous provided
it is continuous at right-dense points in $\mathbb{T}$ and its left-sided limits exists
(finite) at all left-dense points in $\mathbb{T}$.
\end{df}

The set of all rd-continuous functions $f:\mathbb{T} \rightarrow \mathbb{R}$
is denoted by $C_{rd} = C_{rd}(\mathbb{T}) = C_{rd}(\mathbb{T},\mathbb{R})$.
The set of functions $f:\mathbb{T} \rightarrow \mathbb{R}$ that are
$\Delta$-differentiable and whose derivative is rd-continuous is denoted by
$C^{1}_{rd}=C_{rd}^{1}(\mathbb{T})=C^{1}_{rd}(\mathbb{T},\mathbb{R})$.

\begin{tw}[Theorem~1.74 of \cite{BohnerDEOTS}]
Every rd-continuous function $f$ has an antiderivative $F$.
In particular, if $t_{0}\in\mathbb{T}$, then $F$ defined by
$$
F(t):=\int\limits_{t_{0}}^{t} f(\tau)\Delta \tau, \quad t\in\mathbb{T},
$$
is an antiderivative of $f$.
\end{tw}

\begin{df}
Let $\mathbb{T}$ be a time scale and $a,b\in\mathbb{T}$.
If $f:\mathbb{T}^{\kappa} \rightarrow \mathbb{R}$ is a rd-continuous
function and $F:\mathbb{T} \rightarrow \mathbb{R}$
is an antiderivative of $f$, then the $\Delta$-integral is defined by
$$
\int\limits_{a}^{b} f(t)\Delta t := F(b)-F(a).
$$
\end{df}

\begin{ex}
\label{int hZ}
Let $a,b\in\mathbb{T}$ and $f:\mathbb{T} \rightarrow \mathbb{R}$ be rd-continuous.
If $\mathbb{T}=\mathbb{R}$, then
\begin{equation*}
\int\limits_{a}^{b}f(t)\Delta t=\int\limits_{a}^{b}f(t)dt,
\end{equation*}
where the integral on the right side is the usual Riemann integral.
If $\mathbb{T}=h\mathbb{Z}$, $h>0$, then
\begin{equation*}
\int\limits_{a}^{b}f(t)\Delta t
=
\begin{cases}
\sum\limits_{k=\frac{a}{h}}^{\frac{b}{h}-1}f(kh)h, & \hbox{ if } a<b, \\
0, & \hbox{ if } a=b,\\
-\sum\limits_{k=\frac{b}{h}}^{\frac{a}{h}-1}f(kh)h, & \hbox{ if } a>b.
\end{cases}
\end{equation*}
\end{ex}

\begin{tw}[Theorem~1.75 of \cite{BohnerDEOTS}]
\label{eqDelta1}
If $f\in C_{rd}$ and $t\in \mathbb{T}^{\kappa}$, then
\begin{equation*}
\int\limits_{t}^{\sigma(t)}f(\tau)\Delta \tau=\mu(t)f(t).
\end{equation*}
\end{tw}

\begin{tw}[Theorem ~1.77 of \cite{BohnerDEOTS}]
\label{intpropdelta}
If $a,b\in\mathbb{T}$, $a\leqslant c \leqslant b$,
$\alpha\in\mathbb{R}$, and $f,g \in C_{rd}(\mathbb{T}, \mathbb{R})$, then:
\begin{enumerate}
\item
$\int\limits_{a}^{b}(f(t)+g(t))\Delta t
=\int\limits_{a}^{b} f(t)\Delta t+\int\limits_{a}^{b}g(t)\Delta t$,
\item
$\int\limits_{a}^{b}\alpha f(t)\Delta t=\alpha \int\limits_{a}^{b} f(t)\Delta t$,
\item
$\int\limits_{a}^{b}f(t)\Delta t
=-\int\limits_{b}^{a}f(t)\Delta t$,
\item
$\int\limits_{a}^{b} f(t)\Delta t
=\int\limits_{a}^{c} f(t)\Delta t
+\int\limits_{c}^{b} f(t)\Delta t$,
\item
$\int\limits_{a}^{a} f(t)\Delta t=0$,
\item
$\int\limits_{a}^{b}f(t)g^{\Delta}(t)\Delta t
=\left.f(t)g(t)\right|^{t=b}_{t=a}
-\int\limits_{a}^{b} f^{\Delta}(t)g^\sigma(t)\Delta t$,
\item
$\int\limits_{a}^{b} f^\sigma(t) g^{\Delta}(t)\Delta t
=\left.f(t)g(t)\right|^{t=b}_{t=a}
-\int\limits_{a}^{b}f^{\Delta}(t)g(t)\Delta t$,
\item
if $f(t)\geqslant0$ for all $a\leqslant t < b$,
then $\int\limits_{a}^{b}f(t)\Delta t \geqslant 0$.
\end{enumerate}
\end{tw}


\subsection{Delta dynamic equations}
\label{equations}

We now recall the definition and main properties of the delta exponential function.
The general solution to a linear and homogenous second-order delta differential
equation with constant coefficients is given.

\begin{df}[Definition~2.25 of \cite{BohnerDEOTS}]
We say that a function $p:\mathbb{T} \rightarrow \mathbb{R}$ is regressive if
$$
1+\mu(t)p(t) \neq 0
$$
for all $t\in\mathbb{T}^{\kappa}$.
The set of all regressive and rd-continuous functions
$f:\mathbb{T} \rightarrow \mathbb{R}$ is denoted by
$\mathcal{R}=\mathcal{R}(\mathbb{T})=\mathcal{R}(\mathbb{T},\mathbb{R})$.
\end{df}

\begin{df}[Definition~2.30 of \cite{BohnerDEOTS}]
If $p\in\mathcal{R}$, then we define the exponential function by
\begin{equation*}
e_{p}(t,s):= exp\left(\int\limits_{s}^{t}\xi_{\mu(\tau)}(p(\tau))\Delta\tau\right),
\quad s,t\in\mathbb{T},
\end{equation*}
where $\xi_{\mu}$ is the cylinder transformation (see \cite[Definition~2.21]{BohnerDEOTS}).
\end{df}

\begin{ex}
\label{ex:16}
Let $\mathbb{T}$ be a time scale, $t_0 \in \mathbb{T}$,
and $\alpha\in\mathcal{R}(\mathbb{T},\mathbb{R})$.
If $\mathbb{T}=\mathbb{R}$, then
$e_{\alpha}(t,t_{0})=e^{\alpha(t-t_{0})}$ for all $t \in\mathbb{T}$.
If $\mathbb{T}=h\mathbb{Z}$, $h>0$, and
$\alpha\in\mathbb{C}\backslash\left\{-\frac{1}{h}\right\}$ is a constant, then
\begin{equation}
\label{exp:in:hZ}
e_{\alpha}(t,t_{0})=\left(1+\alpha h\right)^{\frac{t-t_{0}}{h}}
\hbox{  for all  } t \in \mathbb{T}.
\end{equation}
\end{ex}

\begin{tw}[Theorem~2.36 of \cite{BohnerDEOTS}]
\label{properties_exp_delta}
Let $p,q\in\mathcal{R}$ and $\ominus p(t):=\frac{-p(t)}{1+\mu(t)p(t)}$.
The following holds:
\begin{enumerate}
\item
$e_{0}(t,s)\equiv 1 \hbox{ and } e_{p}(t,t)\equiv 1$;
\item
$e_{p}(\sigma(t),s)=(1+\mu(t)p(t))e_{p}(t,s)$;
\item
$\frac{1}{e_{p}(t,s)}=e_{\ominus p}(t,s)$;
\item
$e_{p}(t,s)=\frac{1}{e_{p}(s,t)}=e_{\ominus p}(s,t)$;
\item
$\left(\frac{1}{e_{p}(t,s)}\right)^{\Delta}
=\frac{-p(t)}{e_{p}^{\sigma}(t,s)}$.
\end{enumerate}
\end{tw}

\begin{tw}[Theorem~2.62 of \cite{BohnerDEOTS}]
Suppose $y^{\Delta}=p(t)y$ is regressive, that is, $p\in\mathcal{R}$.
Let $t_{0}\in\mathbb{T}$ and $y_{0}\in\mathbb{R}$.
The unique solution to the initial value problem
\begin{equation*}
y^{\Delta}(t) = p(t) y(t), \quad y(t_{0})=y_{0},
\end{equation*}
is given by $y(t)=e_{p}(t,t_{0})y_{0}$.
\end{tw}

Let us consider the following linear second-order dynamic
homogeneous equation with constant coefficients:
\begin{equation}
\label{eq1}
y^{\Delta\Delta}+\alpha y^{\Delta}+\beta y=0,
\quad \alpha, \beta \in \mathbb{R}.
\end{equation}
We say that the dynamic equation \eqref{eq1} is regressive if
$1-\alpha\mu(t)+\beta\mu^{2}(t)\neq 0$ for
$t\in\mathbb{T}^{\kappa}$, i.e., $\beta\mu-\alpha\in\mathcal{R}$.

\begin{df}[Definition~3.5 of \cite{BohnerDEOTS}]
Given two delta differentiable functions $y_{1}$ and $y_{2}$,
we define the Wronskian $W(y_{1},y_{2})(t)$ by
$$
W(y_{1},y_{2})(t) := \det \left[
\begin{array}{cc}
y_{1}(t)&y_{2}(t)\\
y_{1}^{\Delta}(t)&y_{2}^{\Delta}(t)
\end{array}\right].
$$
We say that two solutions $y_{1}$ and $y_{2}$
of \eqref{eq1} form a fundamental set of solutions
(or a fundamental system) for \eqref{eq1},
provided $W(y_{1},y_{2})(t)\neq 0$
for all $t\in\mathbb{T}^{\kappa}$.
\end{df}

\begin{tw}[Theorem~3.16 of \cite{BohnerDEOTS}]
\label{fund sys}
If \eqref{eq1} is regressive and $\alpha^{2}-4\beta\neq 0$,
then a fundamental system for \eqref{eq1} is given by
$e_{\lambda_{1}}(\cdot,t_{0})$ and $e_{\lambda_{2}}(\cdot,t_{0})$,
where $t_{0}\in\mathbb{T}^{\kappa}$ and $\lambda_{1}$ and $\lambda_{2}$ are given by
\begin{equation*}
\lambda_{1} := \frac{-\alpha-\sqrt{\alpha^{2}-4\beta}}{2},
\quad \lambda_{2} := \frac{-\alpha+\sqrt{\alpha^{2}-4\beta}}{2}.
\end{equation*}
\end{tw}

\begin{tw}[Theorem~3.32 of \cite{BohnerDEOTS}]
\label{fund sys2}
Suppose that $\alpha^{2}-4\beta < 0$. Define $p=\frac{-\alpha}{2}$
and $q=\frac{\sqrt{4\beta-\alpha^{2}}}{2}$. If $p$ and $\mu\beta-\alpha$
are regressive, then a fundamental system of \eqref{eq1} is given by
$\cos_{\frac{q}{(1+\mu p)}}(\cdot,t_{0})e_{p}(\cdot,t_{0})$
and $\sin_{\frac{q}{(1+\mu p)}}(\cdot,t_{0})e_{p}(\cdot,t_{0})$,
where $t_{0}\in\mathbb{T}^{\kappa}$.
\end{tw}

\begin{tw}[Theorem~3.34 of \cite{BohnerDEOTS}]
\label{fund sys3}
Suppose $\alpha^{2}-4\beta = 0$. Define $p=\frac{-\alpha}{2}$.
If $p\in\mathcal{R}$, then a fundamental system of \eqref{eq1} is given by
$$
e_{p}(t,t_{0})
\quad \hbox{  and   } \quad
e_{p}(t,t_{0})\int\limits_{t_{0}}^{t}\frac{1}{1+p\mu(\tau)}\Delta \tau,
$$
where $t_{0}\in\mathbb{T}^{\kappa}$.
\end{tw}

\begin{tw}[Theorem~3.7 of \cite{BohnerDEOTS}]
\label{general sol}
If functions $y_{1}$ and $y_{2}$ form
a fundamental system of solutions for \eqref{eq1}, then
$y(t)=\alpha y_{1}(t)+ \beta y_{2}(t)$,
where $\alpha,\beta$ are constants, is a general solution to \eqref{eq1},
i.e., every function of this form is a solution to \eqref{eq1}
and every solution of \eqref{eq1} is of this form.
\end{tw}


\subsection{Calculus of variations on time scales}
\label{CV}

Consider the following problem of the calculus
of variations on time scales:
\begin{equation}
\label{problem}
\mathcal{L}(y)=\int\limits_{a}^{b}
L(t,y(t),y^{\Delta}(t))\Delta t \longrightarrow \min
\end{equation}
subject to the boundary conditions
\begin{equation}
\label{bcproblem}
y(a)=y_{a}, \quad\quad y(b)=y_{b},
\end{equation}
where $L:[a,b]_{\mathbb{T}}^{\kappa}\times\mathbb{R}^{2}\rightarrow\mathbb{R}$,
$(t,y,v) \mapsto L(t,y,v)$, is a given function, and $y_{a}$, $y_{b}\in\mathbb{R}$.

\begin{df}
A function $y \in C^{1}_{rd}([a,b]_{\mathbb{T}},\mathbb{R})$ is said to
be an admissible path to problem \eqref{problem} if it satisfies
the given boundary conditions \eqref{bcproblem}.
\end{df}

We assume that $L(t,\cdot,\cdot)$ is differentiable in $(y,v)$;
$L(t,\cdot,\cdot)$, $L_{y}(t,\cdot,\cdot)$ and $L_{v}(t,\cdot,\cdot)$
are continuous at $(y,y^{\Delta})$ uniformly at $t$
and rd-continuously at $t$ for any admissible path;
the functions $L(\cdot,y(\cdot),y^{\Delta}(\cdot))$, $L_{y}(\cdot,y(\cdot),y^{\Delta}(\cdot))$
and $L_{v}(\cdot,y(\cdot),y^{\Delta}(\cdot))$ are
$\Delta$-integrable on $[a,b]$ for any admissible path $y$.

\begin{df}
We say that an admissible function $\hat{y}$ is a local minimizer
to problem \eqref{problem}--\eqref{bcproblem} if there exists $\delta >0$
such that $\mathcal{L}(\hat{y})\le\mathcal{L}(y)$ for all admissible
functions $y\in C^{1}_{rd}$ satisfying the inequality $||y-\hat{y}||<\delta$.
The following norm in $C^{1}_{rd}$ is considered:
\begin{equation*}
||y||:=\sup\limits_{t\in [a,b]^{\kappa}_{\mathbb{T}}} |y(t)|
+\sup\limits_{t\in [a,b]^{\kappa}_{\mathbb{T}}} \left|y^{\Delta}(t)\right|.
\end{equation*}
\end{df}

\begin{tw}[Corollary~1 of \cite{OptCondHigherDelta}]
\label{corE-Leq}
If $y$ is a local minimizer to problem \eqref{problem}--\eqref{bcproblem},
then $y$ satisfies the Euler--Lagrange equation
\begin{equation}
\label{E-L:eq:T}
L_{v}(t,y(t),y^{\Delta}(t))
=\int\limits_{a}^{\sigma(t)}
L_{y}(\tau,y(\tau),y^{\Delta}(\tau))\Delta\tau +c
\end{equation}
for some constant $c\in\mathbb{R}$ and all
$t\in \left[a,b\right]^{\kappa}_{\mathbb{T}}$.
\end{tw}


\section{Economical model}
\label{model}

The inflation rate, $p$, affects decisions of the society regarding consumption and saving,
and therefore aggregated demand for domestic production, which in turn affects the rate
of unemployment, $u$. A relationship between the inflation rate and the rate of unemployment
is described by the Phillips curve, the most commonly used term in the analysis of inflation
and unemployment \cite{Samuelson}. Having a Phillips tradeoff between $u$ and $p$,
what is then the best combination of inflation and unemployment over time?
To answer this question, we follow here the formulations presented in \cite{ChiangEDO,Taylor}.
The Phillips tradeoff between $u$ and $p$ is defined as
\begin{equation}
\label{inflation}
p := -\beta u+\pi , \quad \beta >0,
\end{equation}
where $\pi$ is the expected rate of inflation that is captured by the equation
\begin{equation}
\label{expected}
\pi'=j(p-\pi), \quad 0< j \le 1.
\end{equation}
The government loss function, $\lambda$,
is specified in the following quadratic form:
\begin{equation}
\label{loss function}
\lambda = u^{2} + \alpha p^{2},
\end{equation}
where $\alpha>0$ is the weight attached to government's distaste for
inflation relative to the loss from income deviating from its equilibrium level.
Combining \eqref{inflation} and \eqref{expected},
and substituting the result into \eqref{loss function}, we obtain that
\begin{equation*}
\lambda\left(\pi(t),\pi'(t)\right)
=\left(\frac{\pi'(t)}{\beta j}\right)^{2}
+\alpha \left(\frac{\pi'(t)}{j}+\pi(t)\right)^{2},
\end{equation*}
where $\alpha$, $\beta$, and $j$ are real positive parameters
that describe the relations between all variables that occur in the model \cite{Taylor}.
The problem consists to find the optimal path $\pi$ that minimizes the total social loss
over the time interval $[0, T]$. The initial and the terminal values of $\pi$,
$\pi_{0}$ and $\pi_{T}$, respectively, are given with $\pi_{0},\pi_{T}>0$.
To recognize the importance of the present over the future,
all social losses are discounted to their present values via a positive discount rate $\delta$.
Two models are available in the literature: the \emph{continuous model}
\begin{equation}
\label{total_social_loss_con}
\Lambda_{C}(\pi)=\int\limits_{0}^{T}\lambda(\pi(t),\pi'(t))e^{-\delta t} dt\longrightarrow \min,
\end{equation}
subject to given boundary conditions
\begin{equation}
\label{eq:bc:cdm}
\pi(0)=\pi_{0}, \quad\pi(T)=\pi_{T},
\end{equation}
and the \emph{discrete model}
\begin{equation}
\label{total_social_loss_disc}
\Lambda_{D}(\pi)=\sum\limits_{t=0}^{T-1}
\lambda(\pi(t),\Delta\pi(t))(1+\delta)^{-t}\longrightarrow \min,
\end{equation}
also subject to given boundary conditions \eqref{eq:bc:cdm}.
In both cases \eqref{total_social_loss_con} and \eqref{total_social_loss_disc},
\begin{equation}
\label{def:lambda}
\lambda(t,\pi,\upsilon) := \left(\frac{\upsilon}{\beta j}\right)^{2}
+\alpha\left(\frac{\upsilon}{j}+\pi\right)^{2}.
\end{equation}
Here we propose the more general \emph{time-scale model}
\begin{equation}
\label{total:social:loss:scale}
\Lambda_{\mathbb{T}}(\pi)=\int\limits_{0}^{T}\lambda(t,\pi(t),\pi^{\Delta}(t))
e_{\ominus\delta}(t,0)\Delta t\longrightarrow \min
\end{equation}
subject to boundary conditions \eqref{eq:bc:cdm}
and with $\lambda$ defined by \eqref{def:lambda}.
Clearly, the time-scale model includes both the discrete
and continuous models as special cases:
our time-scale functional \eqref{total:social:loss:scale}
reduces to \eqref{total_social_loss_con} when $\mathbb{T} = \mathbb{R}$
and to \eqref{total_social_loss_disc} when $\mathbb{T} = \mathbb{Z}$.


\section{Main results}
\label{main:results}

Standard dynamic economic models are set up in either continuous or discrete time.
Since time scale calculus can be used to model dynamic processes whose time domains
are more complex than the set of integers or real numbers, the use of time scales in
economy is a flexible and capable modelling technique.
In this section we show the advantage of using \eqref{total:social:loss:scale}
with the periodic time scale. We begin by obtaining in Section~\ref{main:theory}
a necessary and also a sufficient optimality condition
for our economic model \eqref{total:social:loss:scale}:
Theorem~\ref{cor1} and Theorem~\ref{global}, respectively.
For $\mathbb{T} = h\mathbb{Z}$, $h > 0$, the explicit solution
$\hat{\pi}$ to the problem \eqref{total:social:loss:scale}
subject to \eqref{eq:bc:cdm} is given (Theorem~\ref{th:delf}).
Afterwards, we use such results with empirical data
(Section~\ref{main:empirical}).


\subsection{Theoretical results}
\label{main:theory}

Let us consider the problem
\begin{equation}
\label{mainProblem}
\mathcal{L}(\pi)=\int\limits_{0}^{T}L(t,\pi(t),\pi^{\Delta}(t))\Delta t\longrightarrow \min
\end{equation}
subject to boundary conditions
\begin{equation}
\label{boun:con}
\pi(0)=\pi_{0}, \quad \pi(T)=\pi_{T}.
\end{equation}
As explained in Section~\ref{model}, we are particularly interested
in the situation where
\begin{equation}
\label{eq:pii}
L(t,\pi(t),\pi^{\Delta}(t))=\left[\left(\frac{\pi^{\Delta}(t)}{\beta j}\right)^{2}
+\alpha\left(\frac{\pi^{\Delta}(t)}{j}+\pi(t)\right)^{2}\right]e_{\ominus\delta}(t,0).
\end{equation}
For simplicity, in the sequel we use the notation
$[\pi](t) := (t,\pi(t),\pi^{\Delta}(t))$.

\begin{tw}
\label{cor1}
If $\hat{\pi}$ is a local minimizer to problem \eqref{mainProblem}--\eqref{boun:con}
and the graininess function $\mu$ is a $\Delta$-differentiable function
on $[0,T]^{\kappa}_{\mathbb{T}}$, then $\hat{\pi}$ satisfies the Euler--Lagrange equation
\begin{equation}
\label{E-LDelta}
\left(L_{v}[\pi](t)\right)^{\Delta}
=\left(1+\mu^{\Delta}(t)\right) L_{y}[\pi](t)
+\mu^{\sigma}(t)\left(L_{y}[\pi](t)\right)^{\Delta}
\end{equation}
for all $t\in\ \left[0,T\right]^{\kappa^{2}}_{\mathbb{T}}$.
\end{tw}

\begin{proof}
If $\hat{\pi}$ is a local minimizer to \eqref{mainProblem}--\eqref{boun:con},
then, by Theorem~\ref{corE-Leq}, $\hat{\pi}$ satisfies the following equation:
\begin{equation*}
L_{v}[\pi](t)=\int\limits_{0}^{\sigma(t)}L_{y}[\pi](\tau)\Delta\tau+c.
\end{equation*}
Using the properties of the $\Delta$-integral (see Theorem~\ref{eqDelta1}),
we can write that $\hat{\pi}$ satisfies
\begin{equation}
\label{eq:aux1}
L_{v}[\pi](t)=\int\limits_{0}^{t}L_{y}[\pi](\tau)\Delta\tau +\mu(t)L_{y}[\pi](t) + c.
\end{equation}
Taking the $\Delta$-derivative to both sides of \eqref{eq:aux1},
we obtain equation \eqref{E-LDelta}.
\end{proof}

Using Theorem~\ref{cor1}, we can immediately write the classical
Euler--Lagrange equations for the continuous \eqref{total_social_loss_con}
and the discrete \eqref{total_social_loss_disc} models.

\begin{ex}
\label{E-L_con}
Let $\mathbb{T} = \mathbb{R}$. Then, $\mu \equiv 0$ and
\eqref{E-LDelta} with the Lagrangian \eqref{eq:pii} reduces to
\begin{equation}
\label{eq:EL:ex25}
\left(1+\alpha\beta^{2}\right)\pi^{\prime\prime}(t)
-\delta\left(1+\alpha\beta^{2}\right)\pi^{\prime}(t)
-\alpha j \beta^{2}\left(\delta+j\right)=0.
\end{equation}
This is the Euler--Lagrange equation
for the continuous model \eqref{total_social_loss_con}.
\end{ex}

\begin{ex}
\label{E-L disc}
Let $\mathbb{T} = \mathbb{Z}$. Then, $\mu \equiv 1$ and
\eqref{E-LDelta} with the Lagrangian \eqref{eq:pii} reduces to
\begin{equation}
\label{eq:EL:ex26}
\left(\alpha j\beta^{2}-\alpha\beta^{2}-1\right)\Delta^{2}\pi(t)
+\left(\alpha j^{2}\beta^{2}+\delta\alpha\beta+\delta\right) \Delta\pi(t)
+\alpha j\beta^{2} \left(\delta+j\right)\pi(t)=0.
\end{equation}
This is the Euler--Lagrange equation for the discrete model
\eqref{total_social_loss_disc}.
\end{ex}

\begin{cor}
\label{cor:ThZ}
Let $\mathbb{T} = h\mathbb{Z}$, $h > 0$, $\pi_{0}, \pi_{T} \in\mathbb{R}$, and $T = N h$
for a certain integer $N > 2 h$. If $\hat{\pi}$ is a solution to the problem
\begin{gather*}
\Lambda_{h}(\pi)=\sum\limits_{t=0}^{T-h}L(t,\pi(t),\pi^{\Delta}(t))h \longrightarrow \min,\\
\pi(0)=\pi_{0},\quad \pi(T)=\pi_{T},
\end{gather*}
then $\hat{\pi}$ satisfies the Euler--Lagrange equation
\begin{equation}
\label{eq:cor}
\left(L_{v}[\pi](t)\right)^{\Delta}=L_{y}[\pi](t)+h\left(L_{y}[\pi](t)\right)^{\Delta}
\end{equation}
for all $t\in \{0,\ldots, T-2h\}$.
\end{cor}

\begin{proof}
Follows from Theorem~\ref{cor1} by choosing $\mathbb{T}$
to be the periodic time scale $h\mathbb{Z}$.
\end{proof}

\begin{ex}
\label{eq:exQNI}
The Euler--Lagrange equation for problem
\eqref{total:social:loss:scale} on $\mathbb{T} = h\mathbb{Z}$ is given by \eqref{eq:cor}:
\begin{equation}
\label{E-LeqhZ}
(1+\alpha\beta^{2}-\alpha\beta^{2}j h)\pi^{\Delta\Delta}
+(-\delta-\alpha\beta^{2}\delta
-\alpha\beta^{2}j^{2}h) \pi^{\Delta}
+ (-\alpha\beta^{2}\delta j-\alpha\beta^{2}j^{2})\pi = 0.
\end{equation}
Assume that $1+\alpha\beta^{2}-\alpha\beta^{2}j h\neq 0$.
Then equation \eqref{E-LeqhZ} is regressive and we can use
the well known theorems in the theory of dynamic equations on time scales
(see Section~\ref{equations}), in order to find its general solution.
Introducing the quantities
\begin{equation}
\label{eq:O:A:B}
\Omega := 1+\alpha\beta^{2}-\alpha\beta^{2}jh,
\quad A := -\left(\delta+\alpha\beta^{2}\delta+\alpha\beta^{2}j^{2}h\right),
\quad B := \alpha\beta^{2}j(\delta +j),
\end{equation}
we rewrite equation \eqref{E-LeqhZ} as
\begin{equation}
\label{eqConstDelta}
\pi^{\Delta\Delta}+\frac{A}{\Omega}\pi^{\Delta}-\frac{B}{\Omega}\pi=0.
\end{equation}
The characteristic equation for \eqref{eqConstDelta} is
$$
\varphi(\lambda)=\lambda^{2}+\frac{A}{\Omega}\lambda-\frac{B}{\Omega}=0
$$
with determinant
\begin{equation}
\label{determinant}
\zeta=\frac{A^{2}+4B\Omega}{\Omega^{2}}.
\end{equation}
In general we have three different cases depending on the sign of the determinant $\zeta$:
$\zeta >0$, $\zeta=0$ and $\zeta <0$. However, with our assumptions on the parameters, simple computations
show that the last case cannot occur. Therefore, we consider the two possible cases:
\begin{enumerate}
\item If $\zeta>0$, then we have two different characteristic roots:
\begin{equation*}
\lambda_{1}=\frac{-A+\sqrt{A^{2}+4B\Omega}}{2\Omega} >0
\hbox{  and  }\lambda_{2}=\frac{-A-\sqrt{A^{2}+4B\Omega}}{2\Omega}<0,
\end{equation*}
and by Theorem~\ref{fund sys} and Theorem~\ref{general sol} we get that
\begin{equation}
\label{general solution}
\pi(t)=C_{1}e_{\lambda_{1}}(t,0)+C_{2}e_{\lambda_{2}}(t,0)
\end{equation}
is the general solution to \eqref{eqConstDelta},
where $C_{1}$ and $C_{2}$ are constants
determined using the given boundary conditions \eqref{eq:bc:cdm}.
Using \eqref{exp:in:hZ}, we rewrite \eqref{general solution} as
\begin{equation*}
\pi(t)=C_{1}\left(1+\lambda_{1}h\right)^{\frac{t}{h}}
+C_{2}\left(1+\lambda_{2}h\right)^{\frac{t}{h}}.
\end{equation*}

\item If $\zeta=0$, then by Theorems~\ref{fund sys3} and \ref{general sol} we get that
\begin{equation}
\label{general solution2}
\pi(t)=K_{1}e_{p}(t,0)+K_{2}e_{p}(t,0)\int\limits_{0}^{t}\frac{\Delta\tau}{1+p\mu(\tau)}
\end{equation}
is the general solution to \eqref{eqConstDelta},
where $K_{1}$ and $K_{2}$ are constants, determined using
the given boundary conditions \eqref{eq:bc:cdm},
and $p=-\frac{A}{2 \Omega} \in \mathcal{R}$.
Using Example~\ref{int hZ} and \eqref{exp:in:hZ},
we rewrite \eqref{general solution2} as
$$
\pi(t)= K_{1} \left(1-\frac{A}{2\Omega}h\right)^{\frac{t}{h}}
+K_{2} \left(1-\frac{A}{2\Omega}h\right)^{\frac{t}{h}}\frac{2\Omega t}{2\Omega-Ah}.
$$
\end{enumerate}
\end{ex}

In certain cases one can show that the Euler--Lagrange
extremals are indeed minimizers. In particular,
this is true for the Lagrangian \eqref{eq:pii} under study.
We recall the notion of jointly convex function
(cf., e.g., \cite[Definition~1.6]{book:MT}).

\begin{df}
\label{def:conv}
Function $(t,u,v) \mapsto L(t,u,v)\in C^1\left([a,b]_\mathbb{T}\times\mathbb{R}^{2}; \mathbb{R}\right)$
is jointly convex in $(u,v)$ if
\begin{equation*}
L(t,u+u_0,v+v_0)-L(t,u,v) \geq \partial_2 L(t,u,v)u_0 +\partial_{3}L(t,u,v) v_0
\end{equation*}
for all $(t,u,v)$, $(t,u+u^0,v+v^0) \in [a,b]_\mathbb{T} \times \mathbb{R}^{2}$.
\end{df}

\begin{tw}
\label{global}
Let $(t,u,v) \mapsto L(t,u,v)$ be jointly convex with respect to $(u,v)$
for all $t\in [a,b]_{\mathbb{T}}$. If $\hat{y}$ is a solution to the Euler--Lagrange
equation \eqref{E-L:eq:T}, then $\hat{y}$ is a global minimizer
to \eqref{problem}--\eqref{bcproblem}.
\end{tw}

\begin{proof}
Since $L$ is jointly convex with respect to $(u,v)$ for all $t\in [a,b]_{\mathbb{T}}$,
\begin{multline*}
\mathcal{L}(y)-\mathcal{L}(\hat{y})
=\int\limits_{a}^{b}[L(t,y(t),y^{\Delta}(t))-L(t,\hat{y}(t),\hat{y}^{\Delta}(t))]\Delta t\\
\ge\int\limits_{a}^{b}\left[\partial_{2}L(t,\hat{y}(t),\hat{y}^{\Delta}(t))
\cdot(y(t)-\hat{y}(t))+\partial_{3}L(t,\hat{y}(t),\hat{y}^{\Delta}(t))
\cdot(y^{\Delta}(t)-\hat{y}^{\Delta}(t))\right]\Delta t
\end{multline*}
for any admissible path $y$. Let $h(t) := y(t)-\hat{y}(t)$.
Using boundary conditions \eqref{bcproblem}, we obtain that
\begin{equation*}
\begin{split}
\mathcal{L}(y)-\mathcal{L}(\hat{y})
&\ge \int\limits_{a}^{b} h^{\Delta} (t)\left[-\int\limits_{a}^{\sigma(t)}
\partial_{2}L(\tau,\hat{y}(\tau),\hat{y}^{\Delta}(\tau))\Delta \tau
+\partial_{3}L(t,\hat{y}(t),\hat{y}^{\Delta}(t))\right]\Delta t\\
&\qquad +h(t)\int\limits_{a}^{b}\partial_{2}L(t,\hat{y}(t),\hat{y}^{\Delta}(t))\Delta t|_{a}^{b}\\
&=\int\limits_{a}^{b} h^{\Delta} (t)\left[-\int\limits_{a}^{\sigma(t)}
\partial_{2}L(\tau,\hat{y}(\tau),\hat{y}^{\Delta}(\tau))\Delta \tau
+\partial_{3}L(t,\hat{y}(t),\hat{y}^{\Delta}(t))\right]\Delta t.
\end{split}
\end{equation*}
From \eqref{E-L:eq:T} we get
$$
\mathcal{L}(y)-\mathcal{L}(\hat{y})\ge\int\limits_{a}^{b}h^{\Delta}(t) c\Delta t=0
$$
for some $c\in\mathbb{R}$. Hence, $\mathcal{L}(y)-\mathcal{L}(\hat{y})\ge 0$.
\end{proof}

Combining Examples~\ref{ex:16} and \ref{eq:exQNI} and Theorem~\ref{global},
we obtain the central result to be applied in Section~\ref{main:empirical}.

\begin{tw}[Solution to the total social loss problem of the calculus of variations
in the time scale $\mathbb{T} = h\mathbb{Z}$, $h > 0$]
\label{th:delf}
Let us consider the economic problem
\begin{equation}
\label{functional hZ}
\begin{gathered}
\Lambda_{h}(\pi)=\sum\limits_{t=0}^{T-h}\left[
\left(\frac{\pi^{\Delta}(t)}{\beta j}\right)^{2}
+\alpha\left(\frac{\pi^{\Delta}(t)}{j}+\pi(t)\right)^{2}\right]
\left(1-\frac{h \delta}{1+h\delta} \right)^{\frac{t}{h}} h  \longrightarrow \min,\\
\pi(0)=\pi_{0},\quad \pi(T)=\pi_{T},
\end{gathered}
\end{equation}
discussed in Section~\ref{model}
with $\mathbb{T} = h\mathbb{Z}$, $h > 0$, and the delta derivative
given by \eqref{eq:delta:der:h}. More precisely, let $T = N h$
for a certain integer $N > 2 h$, $\alpha, \beta, \delta, \pi_{0}, \pi_{T} \in\mathbb{R}^+$,
and $0 < j \le 1$ be such that $h > 0$ and $1+\alpha\beta^{2}-\alpha\beta^{2}j h\neq 0$.
Let $\Omega$, $A$ and $B$ be given as in \eqref{eq:O:A:B}.
\begin{enumerate}
\item If $A^{2}+4B\Omega > 0$, then
the solution $\hat{\pi}$ to problem \eqref{functional hZ} is given by
\begin{equation}
\label{eq:exp:rt:delf}
\hat{\pi}(t)=C\left(1-\frac{A-\sqrt{A^{2}+4B\Omega}}{2\Omega}h\right)^{\frac{t}{h}}
+(\pi_0 - C)\left(1-\frac{A+\sqrt{A^{2}+4B\Omega}}{2\Omega}h\right)^{\frac{t}{h}},
\end{equation}
$t\in \{0,\ldots, T-2h\}$, where
$$
C := \frac{\pi_T-\pi_0 \left(\frac {2\,\Omega-hA
-h\sqrt{{A}^{2}+4\,B\Omega}}{2\Omega} \right) ^{{\frac {T}{h}}}}{\left(
\frac {2\,\Omega-hA+h\sqrt{{A}^{2}+4\,B\Omega}}{2\Omega} \right)^{{\frac {T}{h}}}
- \left(\frac {2\,\Omega-hA-h\sqrt {{A}^{2}+4\,B\Omega}}{2\Omega} \right)^{{\frac {T}{h}}}}.
$$
\item If $A^{2}+4B\Omega = 0$, then
the solution $\hat{\pi}$ to problem \eqref{functional hZ} is given by
\begin{equation}
\label{eq:exp:rt:delf2}
\hat{\pi}(t)=\left(1-\frac{A}{2\Omega}h\right)^{\frac{t}{h}}\pi_{0}
+\left(1-\frac{A}{2\Omega}h\right)^{\frac{t}{h}}\left[\pi_{T}\left(
\frac{2\Omega}{2\Omega-Ah}\right)^{\frac{T}{h}}-\pi_{0}\right]\frac{t}{T},
\end{equation}
\end{enumerate}
$t\in \{0,\ldots, T-2h\}$.
\end{tw}

\begin{proof}
From Example~\ref{eq:exQNI}, $\hat{\pi}$ satisfies the Euler--Lagrange equation for problem
\eqref{functional hZ}. Moreover, the Lagrangian of functional $\Lambda_{h}$ of \eqref{functional hZ}
is a convex function because it is the sum of convex functions. Hence, by Theorem~\ref{global},
$\hat{\pi}$ is a global minimizer.
\end{proof}


\subsection{Empirical results}
\label{main:empirical}

We have three forms for the total social loss: continuous \eqref{total_social_loss_con},
discrete  \eqref{total_social_loss_disc}, and on a time scale $\mathbb{T}$ \eqref{total:social:loss:scale}.
Our idea is to compare the implications of one model with those of another using empirical data:
the rate of inflation $p$ from \cite{rateinf} and the rate of unemployment $u$ from \cite{rateunemp},
which were being collected each month in the USA over 11 years, from 2000 to 2010.
We consider the coefficients
$$
\beta := 3, \quad j := \frac{3}{4},\quad \alpha := \frac{1}{2},\quad \delta := \frac{1}{4},
$$
borrowed from \cite{ChiangEDO}. Therefore,
the time-scale total social loss functional for one year is
\begin{equation}
\label{eq2}
\Lambda_{\mathbb{T}}(\pi) = \int\limits_{0}^{11}\left[\frac{16}{9}\left(\pi^{\Delta}(t)\right)^{2}
+\frac{1}{2}\left(\frac{4}{3}\pi^{\Delta}(t)
+\pi(t)\right)^{2}\right]e_{\ominus\frac{1}{4}}(t,0) \Delta t.
\end{equation}
Empirical values $\pi_{E}$ of the expected rate of inflation, $\pi$, for all months in each year,
are calculated using \eqref{inflation} and appropriate values of $p$ and $u$ \cite{rateinf,rateunemp}.
In the sequel, the boundary conditions $\pi(0)$ and $\pi(11)$ will be selected from empirical
data in January and December, respectively.
We shall compare the minimum values of the total social loss functional \eqref{eq2}
obtained from continuous and discrete models and the value for empirical data,
i.e., the value of the discrete functional $\Lambda_{D}(\pi_E) =: \Lambda_{E}$
computed with empirical data $\pi_E$.

In the continuous case we use the Euler--Lagrange equation \eqref{eq:EL:ex25}
with appropriate boundary conditions in order to find the optimal
path that minimizes $\Lambda_{C}$ over each year.
Then, we calculate the optimal values of $\Lambda_{C}$ for each year
(see 2nd column of Table~\ref{tbl:1}). In the 3rd column of Table~\ref{tbl:1}
we collect empirical values of total social loss $\Lambda_{E}$ for each year,
which are obtained by \eqref{total_social_loss_disc} from empirical data.
We find the optimal path that minimizes $\Lambda_{D}$ over each year using
the Euler--Lagrange equation \eqref{eq:EL:ex26}
with appropriate boundary conditions. The optimal values of $\Lambda_{D}$
for each year are given in the 5th column of Table~\ref{tbl:1}.
The paths obtained from the three approaches, using empirical data from 2000,
are presented in Figure~\ref{fig:1}.
\begin{figure}
\begin{center}
\includegraphics[scale=0.4,angle=-90]{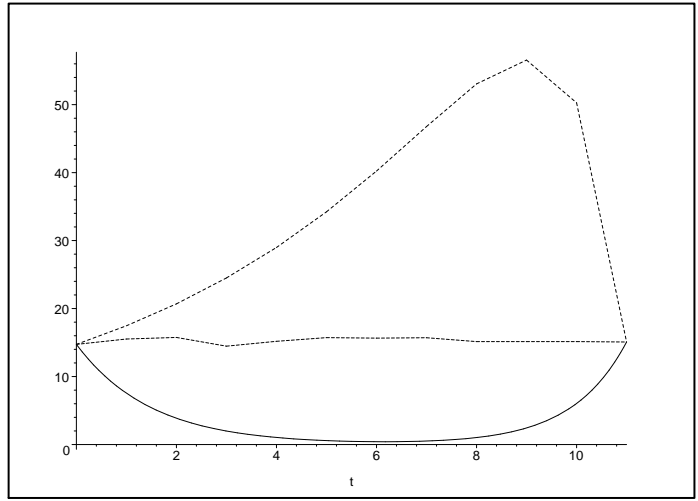}
\end{center}
\caption{\label{fig:1}The expected rate of inflation $\hat{\pi}(t)$
during the year of 2000 in USA, obtained from the classical discrete model \eqref{total_social_loss_disc}
(upper function) and the classical continuous model \eqref{total_social_loss_con}
(lower function), with boundary conditions \eqref{eq:bc:cdm} from January ($t=0$)
and December ($t=11$), together with the empirical rate of inflation
with real data from 2000 \cite{rateinf,rateunemp} (function in the middle).}
\end{figure}
The implications obtained from the three methods in a fixed year are very different,
independently of the year we chose. Table~\ref{tbl:2} shows the relative errors between
$\Lambda_{C}$ and $\Lambda_{E}$ (the 3rd column), $\Lambda_{D}$ and $\Lambda_{E}$
(the 4th column). Our research was motivated by these discrepancies. Why are the results so different?
Is it caused by poor design of the model or maybe by something else?

We focus on the data collection time sampling and consider it as a cause of those differences in the results.
There may exist other reasons, but we examine here the data gathering.
Let us consider our time-scale model in which we consider functional \eqref{eq2}
over a periodic time scale $\mathbb{T}=h\mathbb{Z}$.
In each year we change the time scale by changing $h$,
in such a way that the sum in the functional makes sense,
and we are seeking such value of $h$ for which the absolute error between the minimal
values of the functional \eqref{eq2} and $\Lambda_{E}$ is minimal.
In Table~\ref{tbl:1}, the 6th column presents the values of the most appropriate $h$ and the
4th column the minimal values of the total social loss that correspond to them.
Figure~\ref{fig:2} presents the optimal paths for the continuous, discrete and time-scale models
together with the empirical path, obtained using real data from 2000 \cite{rateinf,rateunemp}.
\begin{figure}
\begin{center}
\includegraphics[scale=0.4,angle=-90]{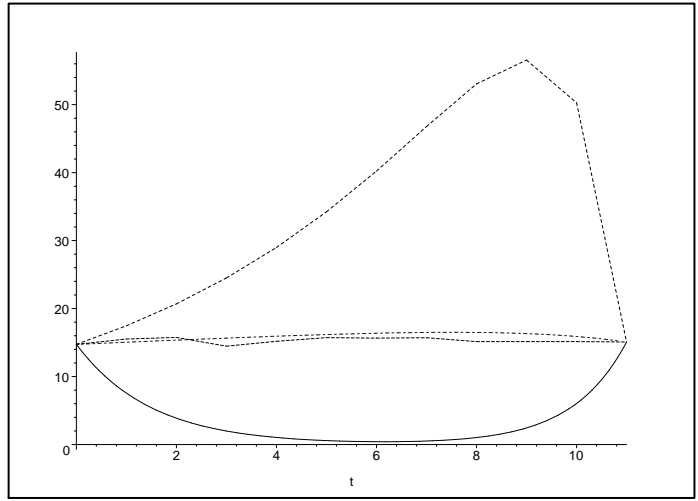}
\end{center}
\caption{\label{fig:2}The three functions of Figure~\ref{fig:1} together with the one
obtained from our time-scale model \eqref{total:social:loss:scale} and Theorem~\ref{th:delf},
illustrating the fact that the expected rate of inflation given by \eqref{eq:exp:rt:delf} with $h = 0,22$
approximates well the empirical rate of inflation.}
\end{figure}
In the 2nd column of Table~\ref{tbl:2} we collect the relative errors between
the minimal values of functional $\Lambda_{E}$ and $\Lambda_{h}$.
\begin{table}[ht]
\begin{center}
\begin{tabular}{|c|c|c|c|c|c|}\hline
year & \multicolumn{4}{c|}{The value of the functional in different approaches} & \ \\ \cline{2-5}
\ &  continuous & empirical & time scales & discrete &  the best $h$  \\ \hline \cline{2-5}
\ & $\Lambda_{C} $ & $\Lambda_{E} $ & $\Lambda_{h} $ & $\Lambda_{D} $ & \   \\ \hline
2000 &37,08888039 & 	457,1493181	 & 487,1508715 & 	2470& 0,22 \\\hline
2001 &52,78839446 & 	522,8060796 & 	536,0298868	 & 3040& 0,11 \\\hline
2002 & 63,88123645 & 	673,399954 & 	663,2573844 & 	3820&  0,11\\\hline
2003 & 62,01139398 & 	811,1909476 & 	853,5383036	 & 4520&  0,2\\\hline
2004 & 61,72908568 & 	703,7663513 & 	699,714732 & 	4130&  0,11\\\hline
2005 & 56,01553586 & 	672,0977499 & 	665,8735854 & 	4060& 0,1 \\\hline
2006 & 45,73885179	 & 592,0374216 & 	594,1793342 & 	3700&  0,1\\\hline
2007 & 53,65457721 & 	505,8743517 & 	511,5351347 & 	2910& 0,1 \\\hline
2008 & 73,4472459	 & 785,9852316 & 	746,8126214 & 	4260&  0,11\\\hline
2009 & 144,2965207 & 	1352,738181 & 	1357,167459 & 	6330&  0,22\\\hline
2010 &153,4630805 & 	1819,572063	 & 1865,77131 & 	11400&  0,1\\\hline
11years & 12,89356177	 &480,5729081	 & 446,1625854	 & 2E+91&  0,11 \\\hline
\end{tabular}
\end{center}
\caption{\label{tbl:1}Comparison of the values of the total social loss functionals in different approaches.}
\end{table}
\begin{table}[ht]
\begin{center}
\begin{tabular}{|c|c|c|c|}\hline
\ & \multicolumn{3}{c|}{Relative error between the empirical value $\Lambda_{E}$ and the result in} \\ \cline{2-4}
year &  time scale $h\mathbb{Z}$ with the best $h$  & continuous approach & discrete classic approach\\\cline{2-4}
\ & $\frac{\Lambda_{h}-\Lambda_{E}}{\Lambda_{E}}$ & $\frac{\Lambda_{C}-\Lambda_{E}}{\Lambda_{E}}$
\ & $\frac{\Lambda_{D}-\Lambda_{E}}{\Lambda_{E}} $\\ \cline{2-4} \hline
2000 & 6,562747053 & 91,88692208 & 440,3048637\\ \hline
2001 &2,529390479 &89,90287288& 481,4775533\\ \hline
2002 & 1,506173195 & 90,51362625& 467,270606\\ \hline
2003 & 5,220393068 & 92,35551208 & 457,2054291 \\ \hline
2004 & 0,575705174 & 91,2287529 & 486,8424929\\ \hline
2005 & 0,926080247 & 91,66556712 & 504,0787967\\ \hline
2006 & 0,361786602 & 92,27433096 & 524,9604949\\ \hline
2007 & 1,119009687 & 89,39369489 & 475,2416564\\\hline
2008 & 4,98388629 & 90,6553911 & 441,9949165\\ \hline
2009 & 0,327430545 & 89,33300451& 367,939775\\ \hline
2010 &2,539017164 & 91,56597952& 526,5209404\\ \hline
11 years & 7,160271027 & 97,31704356 & 4,1617E+90\\ \hline
\end{tabular}
\end{center}
\caption{\label{tbl:2}Relative errors.}
\end{table}


\section{Conclusions}
\label{conclusions}

We introduced a time-scale model to the total social loss over
a certain time interval under study. During examination
of the proposed time-scale model for $\mathbb{T}=h\mathbb{Z}$, $h > 0$,
we changed the graininess parameter. Our goal was to obtain the most similar value
of the total social loss functional $\Lambda_{h}$ to its real value,
i.e., the value from empirical data. We analyzed 11 years
with real data from \cite{rateinf,rateunemp}. With a well-chosen
time scale, we found a small relative error between the real value
of the total social loss and the value obtained
from our time-scale model (see the 2nd column of Table~\ref{tbl:2}).
We conclude that the lack of accurate results
by the classical models arise due to an inappropriate frequency data collection.
Indeed, if one measures the level of inflation and unemployment
about once a week, which is suggested by the values of $h$ obtained
from the time-scale  model, e.g., $h=0,11$ or $h=0,2$ (here $h=1$ corresponds to one month),
the credibility of the results obtained from the classical methods will be much higher.
In other words, similar results to the ones obtained by our time-scale model
can be obtained with the classical models, if a higher frequency of data collection could be used.
In practical terms, however, to collect the levels of inflation and unemployment
on a weekly basis is not realizable, and the calculus of variations
on time scales \cite{Bartos,china-Xuzhou} assumes an important role.


\section*{Acknowledgements}

This work was supported by {\it FEDER} funds through
{\it COMPETE} --- Operational Programme Factors of Competitiveness
(``Programa Operacional Factores de Competitividade'')
and by Portuguese funds through the
{\it Center for Research and Development
in Mathematics and Applications} (University of Aveiro)
and the Portuguese Foundation for Science and Technology
(``FCT --- Funda\c{c}\~{a}o para a Ci\^{e}ncia e a Tecnologia''),
within project PEst-C/MAT/UI4106/2011
with COMPETE number FCOMP-01-0124-FEDER-022690.
Dryl was also supported by FCT through the Ph.D. fellowship
SFRH/BD/51163/2010; Malinowska by Bialystok
University of Technology grant S/WI/02/2011;
and Torres by EU funding under the 7th Framework Programme
FP7-PEOPLE-2010-ITN, grant agreement number 264735-SADCO.

The authors are grateful to two anonymous referees
for valuable suggestions and comments,
which improved the quality of the paper.




\begin{thebibliography}{xx}

\bibitem{almeida:torres}
R. Almeida\ and\ D. F. M. Torres,
Isoperimetric problems on time scales with nabla derivatives,
J. Vib. Control {\bf 15} (2009), no.~6, 951--958.
{\tt arXiv:0811.3650}

\bibitem{Bartos}
Z. Bartosiewicz\ and\ D. F. M. Torres,
Noether's theorem on time scales,
J. Math. Anal. Appl. {\bf 342} (2008), no.~2, 1220--1226.
{\tt arXiv:0709.0400}

\bibitem{BohnerDEOTS}
M. Bohner\ and\ A. Peterson,
{\it Dynamic equations on time scales},
Birkh\"auser Boston, Boston, MA, 2001.

\bibitem{MBbook2001}
M. Bohner\ and\ A. Peterson,
{\it Advances in dynamic equations on time scales},
Birkh\"auser Boston, Boston, MA, 2003.

\bibitem{ChiangEDO}
A. C. Chiang,
{\it Elements of dynamic optimization},
McGraw-Hill, Inc., Singapore, 1992.

\bibitem{OptCondHigherDelta}
R. A. C. Ferreira, A. B. Malinowska\ and\ D. F. M. Torres,
Optimality conditions for the calculus of variations
with higher-order delta derivatives,
Appl. Math. Lett. {\bf 24} (2011), no.~1, 87--92.
{\tt arXiv:1008.1504}

\bibitem{china-Xuzhou}
E. Girejko, A. B. Malinowska\ and\ D. F. M. Torres,
A unified approach to the calculus of variations on time scales,
Proceedings of 2010 CCDC, Xuzhou, China, May 26-28, 2010.
In: IEEE Catalog Number CFP1051D-CDR, 2010, 595--600.
DOI:10.1109/CCDC.2010.5498972
{\tt arXiv:1005.4581}

\bibitem{Girejko}
E. Girejko, A. B. Malinowska\ and\ D. F. M. Torres,
The contingent epiderivative and
the calculus of variations on time scales,
Optimization {\bf 61} (2012), no.~3, 251--264.
{\tt arXiv:1007.0509}

\bibitem{Hilger97}
S. Hilger,
Differential and difference calculus---unified!,
Nonlinear Anal. {\bf 30} (1997), no.~5, 2683--2694.

\bibitem{rateinf}
InflationData.Com, Current inflation,
\url{http://inflationdata.com/Inflation/Inflation_Rate/CurrentInflation.asp}

\bibitem{Malinowska}
A. B. Malinowska\ and\ D. F. M. Torres,
Euler-Lagrange equations for composition functionals
in calculus of variations on time scales,
Discrete Contin. Dyn. Syst. {\bf 29} (2011), no.~2, 577--593.
{\tt arXiv:1007.0584}

\bibitem{book:MT}
A. B. Malinowska\ and\ D. F. M. Torres,
{\it Introduction to the fractional calculus of variations},
Imp. Coll. Press, London, 2012.

\bibitem{Martins}
N. Martins\ and\ D. F. M. Torres,
Generalizing the variational theory on time scales
to include the delta indefinite integral,
Comput. Math. Appl. {\bf 61} (2011), no.~9, 2424--2435.
{\tt arXiv:1102.3727}

\bibitem{moz}
D. Mozyrska\ and\ D. F. M. Torres,
A study of diamond-alpha dynamic equations on regular time scales,
Afr. Diaspora J. Math. (N.S.) {\bf 8} (2009), no.~1, 35--47.
{\tt arXiv:0902.1380}

\bibitem{Samuelson}
P. A. Samuelson\ and\ W. D. Nordhaus,
{\it Economics}, 18th ed.,
McGraw-Hill/Irwin, Boston, MA, 2004.

\bibitem{Taylor}
D. Taylor,
Stopping inflation in the Dornbusch model:
Optimal monetary policies with alternate price-adjustment equations,
Journal of Macroeconomics {\bf 11} (1989), no.~2, 199--216.

\bibitem{rateunemp}
UnemploymentData.com, Unemployment rate,
\url{http://unemploymentdata.com/unemployment-rate}

\end{thebibliography}
\end{document}


\title{Corrigendum to ``A time-scale variational approach\\
to inflation, unemployment and social loss''\thanks{Scheduled to appear
in \emph{Control and Cybernetics} (ISSN: 0324-8569) 43 (2014), no.~4.}}

\author{Monika Dryl$^{1}$\\
\texttt{monikadryl@ua.pt}
\and
Agnieszka B. Malinowska$^{2}$\\
\texttt{a.malinowska@pb.edu.pl}
\and
Delfim F. M. Torres$^{1}$\\
\texttt{delfim@ua.pt}}

\date{$^1$Center for Research and Development in Mathematics and Applications,
Department of Mathematics, University of Aveiro, 3810-193 Aveiro, Portugal\\[0.3cm]
$^2$Faculty of Computer Science, Bialystok University of Technology,\\
15-351 Bia\l ystok, Poland}

\maketitle


\begin{abstract}
The aim of this note is to point out an error in the empirical results of Section~4.2
of [Control Cybernet. 42 (2013), no.~2, 399--418].
\end{abstract}


\bigskip

We regret to announce that there was an error in the \textsf{Maple} computer file that
was used to compute the values for the Example~7 of \cite{MyID:267},
which adversely affects the results of the paper.
Our aim is to clarify here this matter.

\medskip

Example~7 of \cite{MyID:267} is based on real data: the rate of inflation $p$
and the rate of unemployment $u$, taken respectively
from \emph{InflationData.com} and \emph{UnemoloymentData.com},
and which were being collected
each month in the USA over 11 years, from 2000 to 2010.
The time-scale economic model given in Section~4.1 of \cite{MyID:267} is
\begin{equation*}
\mathcal{L}(\pi)=\int\limits_{0}^{T}L\left(t,\pi(t),\pi^{\Delta}(t)\right)
\Delta t\longrightarrow \min
\end{equation*}
subject to boundary conditions
\begin{equation*}
\pi(0)=\pi_{0}, \quad \pi(T)=\pi_{T},
\end{equation*}
where the Lagrangian $L$ is given by
\begin{equation*}
L(t,\pi(t),\pi^{\Delta}(t))=\left[\left(\frac{\pi^{\Delta}(t)}{\beta j}\right)^{2}
+\alpha\left(\frac{\pi^{\Delta}(t)}{j}+\pi(t)\right)^{2}\right]e_{\ominus\delta}(t,0).
\end{equation*}
All parameters of this model, that is, $\beta$, $j$, $\alpha$ and $\delta$,
are specified in Section~4.2 of \cite{MyID:267}.
The characteristic equation of the
Euler--Lagrange equation associated to this problem is given
in Example~7 of \cite{MyID:267}:
\begin{equation}
\label{eq:4}
\lambda^{2}+\frac{A}{\Omega} \lambda -\frac{B}{\Omega} = 0.
\end{equation}
The characteristic roots $\lambda_1$ and $\lambda_2$ of \eqref{eq:4}
are of different signs. They are presented correctly on page 410 of \cite{MyID:267}
but unfortunately they were coded incorrectly in the {\textsf Maple} computer file.
It turns out that the wrong values for the characteristic roots lead to very interesting
results while this is not the case with the correct ones. The wrong values used
in the {\textsf Maple} computer file were
$$
\lambda_{1}(h)=\frac{\frac{11}{8}+\frac{81h}{32}+\frac{1}{2}\sqrt{4\frac{(-\frac{11}{8}
-\frac{81h}{32})^{2}}{(\frac{11}{2}-\frac{27h}{8})^{2}}+\frac{54}{\frac{11}{2}-\frac{27h}{8}}}}{11-\frac{27h}{4}}
$$
and
$$
\lambda_{2}(h)=\frac{\frac{11}{8}+\frac{81h}{32}
-\frac{1}{2}\sqrt{4\frac{(-\frac{11}{8}-\frac{81h}{32})^{2}}{(\frac{11}{2}
-\frac{27h}{8})^{2}}+\frac{54}{\frac{11}{2}-\frac{27h}{8}}}}{11-\frac{27h}{4}}.
$$

